\documentclass[12pt, notitlepage]{article}
\usepackage[utf8]{inputenc}
\usepackage[english]{babel}
\usepackage{amsfonts}
\usepackage{amssymb}
\usepackage{amsmath}
\usepackage{amsthm}
\usepackage{tocloft}
\usepackage{CJK}
\usepackage[english]{babel}
\usepackage{chngcntr}
\counterwithin{equation}{section}

\newtheorem{thm}{Theorem}[section]

\newtheorem{ass}[thm]{\bf {Claim} }

\newtheorem{theorem}[thm]{\bf {Theorem} }

\newtheorem{conj}[thm]{\bf Conjecture}

\makeatletter\@addtoreset{chapter}{part}\makeatother
\newcommand{\xdownarrow}[1]{%
  {\left\downarrow\vbox to #1{}\right.\kern-\nulldelimiterspace}
}

\begin{document}

\title{Fourfolds}

 \author{B. Wang\\
\begin{CJK}{UTF8}{gbsn}
(汪      镔)
\end{CJK}}

\date {}

\maketitle

\begin{abstract}
We have found a ``non-purely-constructive" method of acquiring algebraic cycles involving 
multiple  steps.  This note tries to present the main idea in the last step 
by concentrating on an example of 4-folds. The method demonstrates a contrast to
traditional constructions.

\end{abstract}

\bigskip

  The idea below is based on two main assumptions (that require ``pure" constructions):\par
(1) Lefschetz standard conjecture holds;\par
(2) intersection of currents exists.\bigskip

Let's see the details. Let $X$ be a smooth projective variety of dimension $n$ over $\mathbb C$.
Let $$u\in H^2(X;\mathbb Z)$$ be a hyperplane section class.  We'll
denote the coniveau filtration of coniveau $i$ and degree $2i+k$ 
by $$N^i H^{2i+k}(X)$$ and the linear span
of sub-Hodge structures of coniveau $i$ and degree $2i+k$ by 
$$M^iH^{2i+k}(X).$$
\bigskip

\section{Assumption 1: Lesfchetz standard conjecture} The following assertion will be called the
``generalized Lefschetz standard conjecture''

\begin{conj}
For any whole numbers $p, q, k$ satisfying
$$p+q=n-k, p\leq q$$
the homomorphism
\begin{equation}\begin{array}{ccc}
u^{q-p}_a:  N^{p} H^{2p+k}(X) &\rightarrow & N^{q} H^{2q+k}(X)\\
\alpha &\rightarrow & \alpha\cdot u^{q-p}
\end{array}\end{equation}
is an isomorphism.
\end{conj}
\bigskip

This is easily seen as an equivalent assertion of the famous usual Lefschetz standard conjecture.
The proof of the conjecture is given in [4]. 

 \bigskip

\section{Assumption 2: intersection of currents}

The following is a short description of  ``intersection of currents" ([3]). 
Let $X$ be a compact differential manifold equipped with a de Rham data $\mathcal U$ (the data is irreverent to 
the result).  For any two closed currents $T_1, T_2$, there exists  another current,  
called the intersection of currents (depending on $\mathcal U$), 
\begin{equation}
[T_1\wedge T_2].
\end{equation}
which is a strong limit of a regularization. 
The intersection satisfies many basic properties. Three of them used below:  
\par

 (1) (Cohomologicity) $[T_1\wedge T_2]$ is closed and represents the cup-product \par\hspace{1cc} of the cohomology  of
$T_1, T_2$.

(2)  (Supportivity)
 $$supp([T_1\wedge T_2])\subset supp(T_1)\cap supp(T_2). $$

(3) If we consider two compact manifolds $X, C$. Let $P$ be the projection \par\hspace{1cc}
$C\times X\to X$.
Then up to an exact current,   any closed current $T$ on \par\hspace{1cc} $C\times X$ could satisfy
\begin{equation}
P (supp(T))=supp ( P_\ast (T)).
\end{equation}
\quad\quad \hspace{1cc} (property 3 is irrelevant to the intersection).

\bigskip

\section{ Main theorem}

\begin{theorem}
Let $X$ be a fourfold satisfying that $H^1(X;\mathbb Q)\neq 0$.  The assumptions in section 1, 2 imply the usual Hodge conjecture  on $X$.

\end{theorem}

\begin{proof}
\begin{ass}
There is an algebraic cycle $\langle \beta\rangle \in H^4(X;\mathbb Q)$ for EACH 
non zero $\alpha\in Hdg^4(X) $ such that
the intersection number
\begin{equation}
(\alpha, \langle \beta\rangle)_X\neq 0
.\end{equation}

\end{ass}

The claim 3.2 (for ALL non-zero $\alpha$) is sufficient to prove the usual Hodge conjecture. The statement is called
 algebraic Pincar\'e duality and $\langle \beta\rangle$ an algebraic dual of $\alpha$.

\bigskip

{\it Proof of the claim 3.2}:  We start it by avoiding the middle dimension.
Let $E$ be an elliptic curve and 
$$Y=X\times E.$$
Also let 
$$P: Y\to X$$ be the projection.

Let $a, a'\in H^1(E; \mathbb Q)$ be a standard basis 
i.e. $$a\cup a'=1, a\cup a=a'\cup a'=0.$$ 

Then by taking the tensor product of sub-Hodge structures, we obtain that
\begin{equation}
 \alpha\otimes a'\in M^{2}H^{5}(Y).
\end{equation}

By the Poincar\'e duality, there is a $\theta\in M^{2}H^{5}(Y)$ such that
\begin{equation}
( \alpha\otimes a', \theta)_Y\neq 0.
\end{equation}

Let $\theta$ be generic in $M^{2}H^{5}(Y)$.  Next we would like to claim 
that 
$$\theta\in N^{2}H^{5}(Y).$$
Now we consider the Gysin homomorphism
\begin{equation}\begin{array}{ccc}
P_!: H^{\bullet}(Y;\mathbb Q) &\rightarrow & H^{\bullet-2}(X;\mathbb Q) 
.\end{array}\end{equation}
Notice $$P_!(M^{2}H^{5}(Y))=M^{1} H^{3}(X), $$ and
$M^{1} H^{3}(X)$ contains the nonzero subspace
$$H^1(X;\mathbb Q) u.$$ Thus $P_!$ is not a zero map.
Since $\theta$ is generic in the linear space $M^{2}H^{5}(Y)$, 
$P_!(\theta)\neq 0$.
Now obtain a non-zero cycle 
\begin{equation}
P_!(\theta)\in M^{1} H^{3}(X).
\end{equation}

(Note: $3$ is not the middle dimension).
\medskip

Next argument till (3.13) is the proof of the generalized Hodge conjecture of level 1 on any 4-folds. 
Applying the hard Lefschetz theorem 
\begin{equation}\begin{array}{ccc}
H^3(X;\mathbb Q) &\rightarrow & H^5(X;\mathbb Q)\\
\alpha  &\rightarrow & \alpha\cdot u,
\end{array}\end{equation}

we obtain that 
 $$P_!(\theta) u \in M^{2} H^{5}(X).$$
Let \begin{equation}\begin{array}{ccc}
i: X_{3} &\hookrightarrow & X
\end{array}\end{equation}
be the inclusion map of a smooth hyperplane section $X_{3}$ of $X$.
Then 
\begin{equation}
i_!\circ  i^\ast (P_!(\theta))=P_!(\theta) u.
\end{equation}

Notice \begin{equation}
i^\ast (P_!(\theta))\in M^{1} H^{3}(X_{3}).
\end{equation}
Notice $dim(X_{3})=3$.  By [2],

\begin{equation}
i^\ast (P_!(\theta))\in N^{1} H^{3}(X_{3}).
\end{equation}
(switch the $M$ to the $N$). 
Hence (3.8) implies 
\begin{equation}
P_!(\theta) u\in N^{2} H^{5}(X).
\end{equation}

\bigskip

Now we apply the conjecture 1.1 to obtain that 

\begin{equation}
P_!(\theta) \in N^{1} H^{3}(X).
\end{equation}
( In this step, the generalized Lesfchetz standard conjecture
 converts  $ P_!(\theta) \in M^{1} H^{3}(X)$ to $ P_!(\theta)) \in N^{1} H^{3}(X)$. ).

Next we prove the claim: 
\begin{ass}
\begin{equation}
\theta\in N^{2} H^{5}(Y).
\end{equation}
\end{ass}

{\it Proof  of the claim 3.3}: 

Let $T'$ be a current of a cellular cycle on $Y$ that represents $\theta$. Then 
$P_\ast(T')$ represents the class $ P_!(\theta)$.  
Then by the above argument, 
\begin{equation}
P_\ast(T')=T_{a}+dL
\end{equation}
where $T_a$ is a non-zero current supported on an algebraic set $Z'$ of codimension at least
$1$ (This is only true if $P_!$ is non-zero map).  This means $Z'$ is not an empty set.
 Let $t\in E$ be a point.  The product $dL\otimes {t}$  is a  current on $Y$.
Now we denote the  current  $$T'-dL\otimes {t} $$ by $T_\theta$.
Notice  that according to property 3, section 2, $T_\theta$, by adjusting the exact current in $T'$,  can be chosen such that

\begin{equation}
supp\biggl(P_\ast(T_\theta)\biggr)=P\biggl(supp(T_\theta)\biggr).\end{equation}

Since the pushforward  $$P_\ast (T_\theta) $$
is  supported on the algebraic set $Z'$, $T_\theta$ lies in  the algebraic set 
$$Z=Z'\times E$$ of codimension at least $1$.
Let $\tilde Z$ be a smooth resolution of the scheme $Z$.
We have the following composition map $j$:
\begin{equation}\begin{array}{ccccc}
j: \tilde Z &\rightarrow & Z &\rightarrow & Y
.\end{array}\end{equation}

Next we switch to the classes.
Cor. 8.2.8, [1] says there is an exact sequence 
\begin{equation}\begin{array}{ccccc}
 H^3(\tilde Z;\mathbb Q) &\stackrel{j_!}\rightarrow & H^{5}(Y;\mathbb Q) &\stackrel{r} \rightarrow 
 & H^{5}(Y-Z;\mathbb Q),
\end{array} \end{equation}
where $j_!$ is the Gysin homomorphism.
Since $T_\theta$ is on $Z$, $\theta$ lies in the kernel of $r$.
Then there is 
\begin{equation}
\theta_{\tilde Z}\in M^1 H^{3} (\tilde Z)
\end{equation}
such that 
\begin{equation}
j_! (\theta_{\tilde Z})=\theta 
.\end{equation}
(due to the strictness of the morphism $j_!$ of Hodge structures).
 Notice the dimension of $\tilde Z$ is 4. By the argument from (3.6) to (3.13) which
is the proof of the generalized Hodge conjecture of level 1 on 4-folds,  
\begin{equation}
\theta_{\tilde Z}\in N^1 H^{3}(\tilde Z).
\end{equation}
( In this step, the generalized Hodge conjecture for 4-folds converts the partially Hodge cycle
$\theta_{\tilde Z}\in M^1 H^{3} (\tilde Z)$ to the partially algebraic cycle $\theta_{\tilde Z}\in N^1 H^{3} (\tilde Z)$.).
 Thus we obtain that, 
\begin{equation}
 j_! (\theta_{\tilde Z})=\theta  \in N^2 H^{5}(Y).
\end{equation}

The proof of the claim 3.3 is completed.
\bigskip

Next argument is called ``descending construction".  It extracts  an algebraic cycle from the  partially algebraic cycle
$\theta$. 
By the claim 3.3,  there is an algebraic set $W$  of codimension at least $2$ in $Y$ such that 
$\theta$ is Poincar\'e dual to a singular cycle
$T_\theta$ lying in $W$. 
Next we denote a singular cycle and its associated current by the same letter. 
Applying the K\"unneth decomposition 
 $T_\theta$  must be in the form
of 
\begin{equation} T_\theta=\beta\otimes b+\beta'\otimes b'+\varsigma+ dK\end{equation}
  where $\beta, \beta'$ are closed currents on $X$, representing Hodge cycles,   
$b, b'$  represent  $a, a'$, $dK$ is exact   and
$\varsigma$ is the sum of closed currents in the form $\zeta\otimes c$ with $deg(c)=0, 2$.   
 \par

To show $\beta$ represents an algebraic class, 
we use the developed notion--intersection of currents over $\mathbb R$  in section 2.
Let $b''$ be a closed $1$-current in $E$ such that the intersections satisfy 
$$[b''\wedge b']=0, [b''\wedge b]=\{e\}$$
where $e\in E$. Then the intersection of currents 
\begin{equation}
[(X\otimes b'')\wedge T_\theta]=[(X\otimes b'')\wedge dK]+\beta\times \{e\}
\end{equation}
is a current supported on $W$ (property 2, section 2).  
Let $\tilde W$ be a smooth resolution of the scheme $W$. 
We obtain the diagram
\begin{equation}\begin{array}{ccccc}
H^2(\tilde W;\mathbb Q) &\stackrel{q_!}\rightarrow & H^6(Y;\mathbb Q)&\stackrel{R}
\rightarrow & H^6(Y-W;\mathbb Q)\\
 &\scriptstyle{\nu_!}\searrow & \downarrow\scriptstyle{P_!} &&\\
 & & H^4(X;\mathbb Q) &&,\end{array}\end{equation}
where the top sequence is the Gysin exact sequence, and $\nu_!$, which is a Gysin map,  is the composition 
of Gysin maps $q_!, P_!$.  
Let's observe the class of  $\beta\times \{e\}$.    
By the formula (3.23),  cohomology of $\beta\times \{e\}$ in  $H^6(Y;\mathbb Q)$, denoted by
$$\langle \beta\times \{e\}\rangle$$
( symbol $\langle \cdot\rangle$ denotes the cohomology class)  is in the kernel of $R$. Hence it has a preimage
$$\phi\in H^2(\tilde W;\mathbb Q).$$
Notice that $\phi$ is  Hodge by the strictness of the morphism of Hodge structures (not all preimages are Hodge). Since the $dim(\tilde W)=3$,  the Lefschetz theorem on $(1, 1)$ classes 
implies that $\phi$ is algebraic.  Then the Gysin image $\nu_!(\phi)$ is also algebraic.
By the property 1, section 2, the cohomology of the current $\beta$ 
is exactly $\nu_!(\phi)$ (by the definition of $\phi$). Hence  $\beta$  represents a fundamental class of
an algebraic cycle. 
Then the intersection number 
\begin{equation}
(\alpha\otimes a', \theta)_Y=(\alpha, \langle \beta\rangle )_X\neq 0
\end{equation}
shows $\alpha$ has an algebraic dual $ \langle\beta\rangle$. We complete the proof of claim 3.2.

\bigskip

\end{proof}

\bigskip

PLEASE TURN TO THE NEXT PAGE.

\vfill\eject

 The following is the chart for the process.
$$
 \begin{array}{ccccc}
 Spaces & & \quad & Cohomology\ classes & Levels\\
 &&&& \\
 X && \quad  & \alpha & Hodge\ 0\\
 &&&\\
\bigg\downarrow\scriptstyle{I} &&& \bigg\downarrow &  \bigg\downarrow \\&&&&\\
 X\times E && \quad  &\alpha\otimes a'\Leftrightarrow \theta &Hodge\ 1\\&&&&\\
 \bigg\downarrow\scriptstyle{II}  &&& \bigg\downarrow & \bigg\downarrow\\&&&&\\&&&& \\
X && \quad  &P_!(\theta) \Leftrightarrow P_!(\theta) u  & Hodge\ 1\\&&&\\
 \bigg\downarrow\scriptstyle{III}  &&& \bigg\downarrow & \bigg\downarrow \\&&&&\\
 X\times E && \quad & \theta &Algebraic \ 1\\&&&\\&&&\\
  \bigg\downarrow\scriptstyle{IV}  &&& \bigg\downarrow & \bigg\downarrow \\&&&&\\
  X && \quad &\langle\beta\rangle & Algebraic\ 0

 \end{array}$$

 I: through the  tensor product; II:  through a Gysin map; III: through another type of Gysin map;   
IV: through the descending construction.
\bigskip

\end{document}